\documentclass[a4paper, reqno, 11pt]{amsart}

\makeatletter

\@addtoreset{equation}{section}
\makeatother
\usepackage{setspace}
\usepackage{xcolor}
\usepackage{amsmath}
\usepackage{amssymb}
\usepackage{latexsym}
\usepackage{amsthm}
\usepackage{hyperref}
\usepackage{mathscinet}
\newtheorem{Thm}{Theorem}[section]

\newtheorem{Lem}[Thm]{Lemma}

\newtheorem{Cor}[Thm]{Corollary}

\theoremstyle{definition}

\newtheorem{Rem}[Thm]{Remark}

\begin{document}

\title[]{Characterizations of the Cauchy distribution associated with integral transforms}

\author[K. Okamura]{Kazuki Okamura}
\address{School of General Education, Shinshu University}
\email{kazukio@shinshu-u.ac.jp}

\date{\today}
\maketitle

\begin{abstract}
We give two new simple characterizations of the Cauchy distribution by using the  M\"obius and Mellin transforms.   
They  also yield characterizations of the circular Cauchy distribution and the mixture Cauchy model. 
\end{abstract}

\section{Introduction} 

The Cauchy distribution is a statistical model with a heavy-tailed symmetric distribution. 
We cannot define its expected value and its variance, and it has no moment generating function, due to its heavy tails. 
It also appears in physics, and is called the Lorentz distribution alternatively. 
Characterizations of probability distributions are interesting in itself and useful when we choose a suitable statistical model. 
Various characterizations of the Cauchy distribution have been considered by many authors \cite{Arnold1979,Arnold1990, Bell1985, Chin2020, Dunau1987, Hamedani1993, Hassenforder1988, Knight1976b, Knight1976a, Letac1977, Menon1962, Menon1966, Norton1983, Obretenov1961, Ramachandran1970, Williams1969, Yanushkevichius2007,Yanushkevichius2014}. 

This paper proposes yet another type of characterizations of the Cauchy distribution. 
Our characterizations concern  integral transforms, specifically, the M\"obius and Mellin transforms.    
The M\"obius and Mellin transforms of the Cauchy distribution have somewhat simpler forms than the characteristic function of it, that is, the Fourier transform of it. 
Our proofs utilize some basic facts of complex analysis and functional analysis. 
Furthermore, our approach immediately yields characterizations of the circular Cauchy distribution and the mixture Cauchy model.

This paper adopts McCullagh's parametrization of the Cauchy distribution \cite{McCullagh1996}. 
Let $i$ be the imaginary unit.
For a complex number $\gamma$, let $\textup{Re}(\gamma)$ and $\textup{Im}(\gamma)$ be the real and imaginary parts of $\gamma$ respectively. 
We denote the distribution with density function 
\[ p(x; \gamma) := \frac{\textup{Im}(\gamma)}{\pi} \frac{1}{|x - \gamma|^2}, \ x \in \mathbb{R},    \]
by $C(\gamma), \ \gamma \in \mathbb{H}$, where we let $\mathbb{H} := \{x+yi : y > 0\}$.

This paper is organized as follows. 
In Section 2, we give a characterization of the Cauchy distribution by the M\"obius transforms with an application to a characterization of the circular Cauchy distribution. 
In Section 3, we give a characterization of the Cauchy distribution by the Mellin transforms with an application to a characterization of the mixture Cauchy model.

\section{Characterization by M\"obius transforms}

Hereafter the symbol $E$ denotes the notation of the expectation of random variables and we denote the complex conjugate of a complex number $\gamma$ by $\overline{\gamma}$. %added; 

\begin{Thm}\label{Mobius}
Let $X$ be a real-valued random variable such that there exists  $\alpha \in \mathbb{H}$ such that  
\[ \alpha = \frac{E\left[\dfrac{X}{X - \overline{\gamma}}\right]}{E\left[\dfrac{1}{X - \overline{\gamma}}\right]}\]
for every $\gamma$ in a subset $D$ of $\mathbb{H}$ having a limit point in $\mathbb{H}$.
Then, 
$X$ follows $C(\alpha)$.  
\end{Thm}

Let $C_0 (\mathbb{R})$ be the set of continuous functions vanishing at infinity. 
The following is standard. 
\begin{Lem}\label{Cc}
Let $f \in C_0 (\mathbb{R})$. 
Then, 
\[ \lim_{b \to +0} \sup_{a \in \mathbb{R}} \left| f(a) - \int_{\mathbb{R}} f(x) \frac{b}{\pi((x-a)^2 + b^2)} dx \right| = 0.   \]
\end{Lem}

\begin{proof}
We have that 
\[  \left|f(a) - \int_{\mathbb R} f(x) \frac{b}{\pi((x-a)^2 + b^2)} dx \right|   \le \int_{\mathbb R}  | f(a+t) - f(a) | \frac{b}{\pi(t^2 + b^2)} dt. \]
Since $f \in C_0 (\mathbb{R})$, $f$ is uniformly continuous on $\mathbb{R}$, that is, 
it holds that 
for every  $\epsilon > 0$, there exists $\delta > 0$ such that for every $t \in [-\delta, \delta]$, 
\[  \sup_{a \in \mathbb{R}}  | f(a+t) - f(a) |  \le \epsilon. \]
By this and the fact that $\displaystyle \int_{\mathbb{R}} \frac{b}{\pi((x-a)^2 + b^2)} dx = 1$, 
we have that 
\[  \int_{\mathbb R} \sup_{a \in \mathbb{R}}  | f(a+t) - f(a) |  \frac{b}{\pi(t^2 + b^2)} dt \le \epsilon + 2 \|f\|_{\infty} \int_{\mathbb{R} \setminus [-\delta, \delta]} \frac{b}{\pi(t^2 + b^2)} dt, \]
where $\| f \|_{\infty}$ denotes the supremum norm of $f$.  %corrected-1
Since $\dfrac{b}{\pi(t^2 + b^2)}$ is increasing as a function of $b \in \mathbb{R} \setminus [-\delta, \delta]$, %corrected-2 
by applying the Lebesgue dominated convergence theorem, %corrected-3
we have that 
\[ \lim_{b \to +0}  \int_{\mathbb{R} \setminus [-\delta, \delta]} \frac{b}{\pi(t^2 + b^2)} dt = 0.  \] 
\end{proof}

For $\gamma \in \mathbb{H}$, 
we let $\phi_{\gamma} : \mathbb{H} \to \mathbb{D}$ be the function defined by 
\[ \phi_{\gamma}(z) := \frac{z - \gamma}{z - \overline{\gamma}}, \ \ z \in \mathbb{H}, \]
which is a M\"obius transform and could be regarded as a certain generalization of the Cayley transform.  

\begin{proof}[Proof of Theorem \ref{Mobius}] 
We have that for every $\gamma \in D$, 
\begin{equation}\label{Mobius-mean} 
E\left[ \phi_{\gamma}(X)  \right] = \phi_{\gamma}(\alpha). 
\end{equation}
By the residue theorem, 
we see that \eqref{Mobius-mean} holds for $X$ following $C(\alpha)$. 
  
Let $\mu$ be the Borel probability measure on $\mathbb{R}$ induced by $X$.  
Then, 
\begin{equation}\label{eq-compare} 
\int_{\mathbb{R}} \phi_{\gamma}(x) \mu(dx) = \int_{\mathbb{R}} \phi_{\gamma}(x)  \nu(dx),  
\end{equation}
where $\gamma \in D$ and  we let 
\[ \nu(dx) := \frac{\textup{Im}(\alpha)}{|x - \alpha|^2} dx.  \]

Let $$F_{\mu}(a+bi) := \frac{1}{\pi} \int_{\mathbb{R}} \frac{1}{x - (a+bi)} \mu(dx), \ a+bi \in \mathbb{H}.$$
This is holomorphic on $\mathbb{H}$. 
Let $U_{\mu}(a+bi)$ and $V_{\mu}(a+bi)$ be the real and imaginary parts of $F_{\mu}(a+bi)$ respectively. 
By replacing $\mu$ with $\nu$, we define $F_{\nu}(a+bi)$, $U_{\nu}(a+bi)$ and $V_{\nu}(a+bi)$ in the same manner.  

By comparing the real and imaginary parts of \eqref{eq-compare}, 
we have that 
for every $\gamma = a+bi \in D$, 
\[ \int_{\mathbb{R}} \frac{b}{\pi((x-a)^2 + b^2)} \mu(dx) = \int_{\mathbb{R}}  \frac{b}{\pi((x-a)^2 + b^2)} \nu(dx),  \]
and 
\[ \int_{\mathbb{R}} \frac{x-a}{\pi((x-a)^2 + b^2)} \mu(dx) = \int_{\mathbb{R}}  \frac{x-a}{\pi((x-a)^2 + b^2)} \nu(dx). \]
Therefore, 
 $F_{\mu} = F_{\nu}$ on $D$. 
By applying  the identity theorem for holomorphic functions \cite[Theorem 10.18]{Rudin1987},  
$F_{\mu} = F_{\nu}$ on $\mathbb{H}$.

By this and Fubini's theorem,  
\[ \int_{\mathbb{R}} \int_{\mathbb{R}} \frac{b}{\pi((x-a)^2 + b^2)} f(x) dx  \mu(da) = \int_{\mathbb{R}} V_{\mu}(x+bi) f(x) dx \]
\[ = \int_{\mathbb{R}} V_{\nu}(x+bi)  f(x) dx = \int_{\mathbb{R}} \int_{\mathbb{R}} \frac{b}{\pi((x-a)^2 + b^2)} f(x) dx  \nu(da).  \]

By Lemma \ref{Cc}, we have that 
\[ \int_{\mathbb{R}} f(a)  \mu(da)  = \lim_{b \to +0} \int_{\mathbb{R}} \int_{\mathbb{R}} \frac{b}{\pi((x-a)^2 + b^2)} f(x) dx  \mu(da), \]
and 
\[ \int_{\mathbb{R}} f(a)  \nu(da)  = \lim_{b \to +0} \int_{\mathbb{R}} \int_{\mathbb{R}} \frac{b}{\pi((x-a)^2 + b^2)} f(x) dx  \nu(da). \]

Thus we have that 
\[ \int_{\mathbb{R}} f(a)  \mu(da)  = \int_{\mathbb{R}} f(a)  \nu(da).  \]

Since $\mu$ and $\nu$ are both  regular, by the Riesz-Markov-Kakutani theorem \cite[Theorem 6.19]{Rudin1987}, 
we have that $\mu = \nu$, which means that $X$ follows the Cauchy distribution with parameter $\alpha$.  
\end{proof}

\begin{Rem}\label{upper}
Let $\overline{\mathbb{H}}$ be the closure of $\mathbb{H}$, that is, $\overline{\mathbb{H}} := \{x+yi : y \ge 0\}$. 
Let $F_X (\gamma) := E\left[\dfrac{X}{X - \overline{\gamma}}\right] / E\left[\dfrac{1}{X - \overline{\gamma}}\right]$. 
Since 
\[ F_X (\gamma)  - \overline{F_X (\gamma)} = \dfrac{E\left[ \dfrac{(X-a)^2}{(X-a)^2 + b^2} \right]E\left[ \dfrac{1}{(X-a)^2 + b^2}\right] - E\left[ \dfrac{X-a}{(X-a)^2 + b^2}\right]^2}{\left|E\left[\dfrac{1}{X - \gamma}\right]\right|^2}i,\]
where we let $\gamma = a+bi$, 
it holds that for every $\gamma \in \mathbb{H}$, $F_X (\gamma) \in \overline{\mathbb H}$, and furthermore, it holds that 
$F_X (\gamma) \in \mathbb H$ if and only if the distribution of $X$ is not a point mass.  
It holds that 
$\gamma_n$ is the maximal likelihood estimator of Cauchy samples $\{x_1, \cdots, x_n\}$, 
if and only if 
$\gamma_n = F_X (\gamma_n)$ 
where the expectation  is taken with respect to $\displaystyle \frac{1}{n} \sum_{i=1}^{n} \delta_{x_i}$. 
\end{Rem}

The circular Cauchy distribution, also known as the wrapped Cauchy distribution, appears in the area of directional statistics. 
It is a distribution on the unit circle and is connected with the Cauchy distribution via M\"obius transforms.  
Such connection is considered by  \cite{McCullagh1996}. 
 
Let $\mathbb D := \{z \in \mathbb{C} : |z| < 1\}$. 
The circular-Cauchy distribution $P^{\textup{cc}}_w$ with parameter $w \in \mathbb{D}$ is the continuous distribution on $[0, 2\pi)$  
with density function 
\[  \frac{1}{2\pi} \frac{1 - |w|^2}{|\exp(ix) - w|^2}, \ \ x \in [0, 2\pi). \]

We remark that $\phi_{\gamma}$ is a bijection between $\mathbb H$ and $\mathbb D$, and furthermore its inverse is given by 
\[ \phi_{\gamma}^{-1}(w) = \frac{\gamma - \overline{\gamma}w}{1 - w}, \ w \in \mathbb{D}. \]
We can extend the domain of $\phi_{\gamma}$ to $\overline{\mathbb H}$. 
$\phi_{\gamma}$ defines a bijection between $\mathbb{R}$ and $\{z \in \mathbb{C} : |z| =1, z \ne 1\}$. 

If a random variable $X$ follows the circular-Cauchy distribution $P^{\textup{cc}}_{w}$, 
then, $\phi_{\gamma}^{-1}(\exp(iX))$ follows the Cauchy distribution with parameter $\phi_{\gamma}^{-1}(w)$. 
Therefore, by computations with Theorem \ref{Mobius}, 
we have that 
\begin{Cor}
Let $X$ be a $[0, 2\pi)$-valued random variable such that there exists $w \in \mathbb{D}$ such that 
\[ w = \frac{E\left[ \dfrac{\exp(iX)}{ 1 - \eta \exp(iX)} \right]}{E\left[ \dfrac{1}{1 - \eta  \exp(iX)} \right]} \]
for every $\eta$ in a subset $\widetilde D$ of $\mathbb{D}$ having a limit point in $\mathbb{D}$. 
We also assume that $P(X \ne 0) = 1$. %corrected-4
Then, $X$ follows the circular-Cauchy distribution $P^{\textup{cc}}_{w}$. 
\end{Cor}

Let $\overline{\mathbb{D}}$ be the closure of $\mathbb{D}$, that is, $\overline{\mathbb{D}} := \{z : |z| \le 1\}$. 
Let 
$G_X (\eta) := E\left[ \dfrac{\exp(iX)}{ 1 - \eta \exp(iX)} \right]/E\left[ \dfrac{1}{1 - \eta  \exp(iX)} \right]$. 
Assume that $P(X \ne 0) = 1$. %corrected-4
Then, for every $\gamma \in \mathbb{H}$, 
$\phi^{-1}_{\gamma} \left( G_X (\eta)\right) = F_Y (\phi^{-1}_{\gamma} ( \eta ))$ where we let $Y := \phi_{\gamma}^{-1}(\exp(iX))$. 
Then, by Remark \ref{upper}, it holds that for every $\eta \in \mathbb{D}$, $G_X (\eta) \in \overline{\mathbb D}$, and furthermore, 
$G_X (\eta) \in \mathbb D$ if and only if the distribution of $X$ is not a point mass.

\section{Characterization by Mellin transforms}

We define the logarithm for complex numbers as follows. 
For $z = r \exp(i \theta) \in \mathbb H$ where $r > 0$ and $-\pi \le \theta < \pi$, 
we let 
\[ \log z := \log r + i \theta.  \]
This is holomorphic on $U := \mathbb{C} \setminus \left\{z \in \mathbb{C} | \textup{Re}(z) \le 0, \textup{Im}(z) = 0 \right\}$. 
Then, 
\begin{equation}\label{def-log} 
\log x = \log |x| + i \pi \mathbf{1}_{(-\infty, 0)}(x), \ \ x \in \mathbb{R} \setminus \{0\}, 
\end{equation}
where $\mathbf{1}_{(-\infty, 0)}$ denotes the indicator function of $(-\infty,0)$. 
For every $a \in \mathbb{C}$, we let 
\[ z^{a} := \exp(a \log z), \ \ \ z \in U. \]
For every $a \in \mathbb{C}$, we let $0^a := 0$. 
This definition is also adopted for $a = 0$. 
We remark that $x^{a}$ is {\it not} a real number if $x < 0$ and $a \in \mathbb{R} \setminus \mathbb{Z}$. %corrected-5
For example, $(-8)^{1/3} = -2 \exp(i \pi/3)$. %corrected-5

In this paper, we call $E[X^a]$ the Mellin transform of the random variable $X$. 
We deal with the powers of {\it negative} numbers by allowing the powers to be {\it complex-valued}.  
In this point, our definition of  the powers of  random variables  is different from the one given in Zolotarev \cite[(3.0.4)]{Zolotarev1986}.

\begin{Thm}\label{power} 
Let $X$ be a real-valued random variable such that $E\left[|X|^{\delta}\right] < \infty$  for some $\delta > 0$. 
If it holds that $E[X^a] = \gamma^a, a \in D$ for  a subset $D \subset (0, \delta)$ having a limit point in $(0, \delta)$ and some $\gamma \in \mathbb{H}$, 
then, $X$ follows  $C(\gamma)$. 
\end{Thm}

Our proof of this assertion depends on Galambos and Simonelli \cite[Theorem 1.19]{galambos2004}. 
However their definition of  the Mellin transform of random variables is somewhat different from ours, so we need some arguments. 

\begin{proof}[Proof of Theorem \ref{power}]

\begin{Lem}\label{F-hol}
Let $J := \left\{x+y i : 0 < x < \delta \right\}$ and $ \overline{J}$ be the closure of $J$.  
Let $f(a) := E[X^a], a \in \overline{J}$. 
Then, $f$ is well-defined and continuous on $\overline{J}$ and holomorphic on $J$. 
\end{Lem}

\begin{proof}
Since $|X^a| = |X|^{\textup{Re}(a)}$ and $\textup{Re}(a) \in [0, \delta]$,
we see that $f$ is well-defined and continuous on $\overline{J}$. 

Let $a \in J$ and $h \ne 0$.  
We remark that $E[X^a] = E\left[X^a, \ X \ne 0\right]$. 
Then we have that 
\[ \frac{f(a+h) - f(a)}{h} - E[X^a \log X, X \ne 0] = E\left[X^a \left( \frac{X^h - 1}{h} - \log X \right), \ \  X \ne 0 \right].  \]

If $X \ne 0$, then, 
\[ \left| X^a \left(\frac{X^h - 1}{h} - \log X  \right) \right| \le |X|^{\textup{Re}(a)} \cdot  |h| |\log X|^{2}  \exp(|h \log X|)  \]
\[ \le |h|  |X|^{\textup{Re}(a)}   \left( |\log |X||^{2} + \pi^2 \right)  \exp\left(|h| \left( |\log |X|| + \pi \right) \right).  \]

If $|X| \ge 1$ and $|h| \le (\delta - \textup{Re}(a))/2$, then, 
\[ \left|  X^a \left( \frac{X^h - 1}{h} - \log X \right) \right| \le \delta \exp(\delta \pi) |X|^{\frac{\delta + \textup{Re}(a)}{2}} ((\log |X|)^{2}  + \pi^2). \]

By the assumption, 
\[ E\left[  |X|^{\frac{\delta + \textup{Re}(a)}{2}} \left(1 + (\log |X|)^{2} \right), \ |X| \ge 1  \right] < +\infty. \]

If $|X| \le 1$ and $|h| \le \textup{Re}(a)/2$, then, 
\[ \left|  X^a \left( \frac{X^h - 1}{h} - \log X \right) \right| \le \delta \exp(\delta \pi) |X|^{\frac{\textup{Re}(a)}{2}} ((\log |X|)^{2}  + \pi^2). \]

By the assumption and the fact that $\displaystyle \lim_{x \to +0} x^{\beta} \log x = 0$ for every $\beta > 0$, %corrected-6
\[ E\left[  |X|^{\frac{\textup{Re}(a)}{2}} \left(1 + (\log |X|)^{2} \right), \ |X| \le 1  \right] < +\infty. \]

By the Lebesgue dominated convergence theorem, %corrected-7
\[ E\left[X^a \left( \frac{X^h - 1}{h} - \log X \right), \ \  X \ne 0 \right] \to 0, \ h \to 0. \]
\end{proof}

\begin{Lem}\label{F-rep}
\[ f(a) = \gamma^a, \  \ a \in \overline{J}. \]
\end{Lem}

\begin{proof}  
Let $\widetilde f(a) := \gamma^a, \ a \in \overline{J}$. 
This is holomorphic on $J$. 
By the assumption of Theorem \ref{power}, 
it holds that 
$\widetilde f(a) = f(a), \  a \in D$.  
By Lemma \ref{F-hol}, $f$ is holomorphic on $J$. 
Hence, by the identity theorem for holomorphic functions,  
$\widetilde f(a) = f(a), \  a \in J$. 
Since $f$ and $\widetilde f$ are both continuous on $\overline{J}$, 
we have the assertion. 
\end{proof}

\begin{Lem}\label{g-exp-pre}
Let 
\begin{equation}\label{g-def}
g(a) := E\left[X^a, \ X > 0 \right] + i E\left[(-X)^a, \ X < 0 \right].
\end{equation} 
Then,  $g$ is well-defined and continuous on $\overline{J}$ and holomorphic on $J$. 
Furthermore, 
\begin{equation}\label{g-strip} 
g(a) = r^a \left(\cos(a \theta) - \frac{\sin(a\theta)}{\sin(a\pi)} \cos(a\pi) + i \frac{\sin(a\theta)}{\sin(a\pi)} \right), \ a \in J. 
\end{equation}
\end{Lem}

We remark that \eqref{g-def} is equivalent to the definition of  the Mellin transform of $X$ in \cite[Section 1.3]{galambos2004}. %corrected-8

\begin{proof}
Since $E[|X|^{\delta}] < +\infty$, 
$g$ is well-defined and continuous on $\overline{J}$.  
If $0 < a < \delta$, then, %corrected-9
\begin{equation}\label{f-re-im} 
f(a) = E\left[X^a, \ X > 0 \right] + E\left[(-X)^a, \ X < 0 \right] \cos(a \pi) +  i E\left[(-X)^a, \ X < 0 \right] \sin(a \pi).  
\end{equation}
If $0 < a < \delta$, then, by Lemma \ref{F-rep}, %corrected-9
\begin{equation}\label{x-nega} 
E\left[(-X)^a, \ X < 0 \right] = r^a \frac{\sin(a\theta)}{\sin(a\pi)},  
\end{equation}
where we let $\gamma = r \exp(i \theta)$.  
Hence, as a function of $a$, 
$E\left[(-X)^a, \ X < 0 \right]$ is holomorphic on $J$. 
By using  Lemma \ref{F-rep}, we have that 
\[ E\left[X^a, \ X > 0 \right] + E\left[(-X)^a, \ X < 0 \right] \cos(a \pi) = r^a \cos(a\theta), \ 0 < a < \delta. \]%corrected-9

Therefore we have that as a function of $a$, 
$E\left[X^a, \ X > 0 \right]$ is holomorphic on $J$ and
\begin{equation}\label{x-posi} 
E\left[X^a, \ X > 0 \right] = r^a \left(\cos(a\theta) - \frac{\sin(a\theta)}{\sin(a\pi)}\cos(a\pi) \right). 
\end{equation}
By \eqref{f-re-im}, \eqref{x-nega} and \eqref{x-posi}, 
we have \eqref{g-strip}. 
\end{proof}

Now we return to the proof of Theorem \ref{power}. %corrected-10; position changed
Since $\sin(a\pi) \ne 0$ for $a \in \{yi : y \ne 0\}$ and $\displaystyle \lim_{a \to 0} \frac{\sin(a\theta)}{\sin(a\pi)} = \frac{\theta}{\pi}$, %added; corrected-10
we could continuously extend the function $g$ in  \eqref{g-strip} to the left boundary of $J$, which is the imaginary axis $\{yi : y \in \mathbb{R}\}$. 
%Lem 3.5 deleted added; corrected-10

If $X$ follows the Cauchy distribution $C(\gamma)$, then, by the residue theorem, %corrected-11
\[ E[X^a] = \frac{\textup{Im}(\gamma)}{\pi} \int_{\mathbb{R}}  \frac{x^a}{|x - \gamma|^2} dx = \gamma^a, \ a \in J. \]
Hence if we define $g$ for $X$ following  the Cauchy distribution $C(\gamma)$ in the same manner as in \eqref{g-def}, %corrected-11
then we have the same expression for $g$ as in \eqref{g-strip}, %corrected-11
and in particular, they are identical with each other on  the imaginary axis $\{yi : y \in \mathbb{R}\}$. 
Now Theorem \ref{power} follows from \cite[Theorem 1.19]{galambos2004}. 
\end{proof} 

We also have the following claim which is similar to Theorem \ref{power}. 

\begin{Thm}\label{power-2} 
Let $X$ be a real-valued random variable such that $P(X = 0) = 0$ and  $E\left[|X|^{\delta}\right] + E\left[|X|^{-\delta}\right] < \infty$  for some $\delta > 0$. 
If $E[X^a] = \gamma^a, a \in D$ for  a subset $D \subset (-\delta, \delta)$ having a limit point in $ (-\delta, \delta)$ and some $\gamma \in \mathbb{H}$, 
then, $X$ follows $C(\gamma)$. 
\end{Thm}

The proof of Theorem \ref{power-2} goes in the same manner as in the proof of Theorem \ref{power}. 

\begin{Cor}
Let $X$ be a real-valued random variable such that $P(X = 0) = 0$ and  $E\left[|X|^{\delta}\right] + E\left[|X|^{-\delta}\right] < \infty$  for some $\delta > 0$. Then, \\
(i)  If $E\left[ X^{1/k_n}\right]^{k_n} = \gamma \in \mathbb{H}$ for an infinite increasing sequence $(k_n)_n$,   
then, $X$ follows  $C(\gamma)$. \\
(ii) If $0 < P(X < 0) < 1$ and $E[(\log X)^n] = E[\log X]^n$ for every $n \in \mathbb{N}$, %added; comment-12
then, $X$ follows  $C(E[\log X])$. 
\end{Cor}

We remark that if $0 < P(X < 0) < 1$ and $P(X = 0) = 0$, then, $0 < P(X > 0) < 1$, and furthermore, $X$ is non-atomic. %added; comment-12

\begin{proof}
Assertion (i) follows from Theorem \ref{power-2}. 

(ii) We remark that for $p \in (-\delta, \delta)$, 
\[ E[\exp(|p \log X|)] \le \exp(|p| \pi) E[\exp(|p| |\log |X||)] \le E\left[|X|^{-|p|} + |X|^{|p|}\right] < +\infty. \]
By the Lebesgue convergence theorem and the assumption, we have that for $p \in (-\delta, \delta)$, 
\[ E[X^p] = E\left[ \sum_{n=0}^{\infty} \frac{p^n (\log X)^n}{n!}\right] = \sum_{n=0}^{\infty} \frac{p^n E\left[ (\log X)^n\right]}{n!}  = \exp(p E[\log X]). \]

%added; comment-12
By \eqref{def-log}, we have that 
\[ \exp(E[\log X]) = \exp(E[\log |X|]) \exp(i \pi P(X < 0)).  \]
By the assumption, we have that $\sin(\pi P(X < 0)) > 0$. 
Hence, it holds that  $\exp(E[\log X]) \in \mathbb{H}$. 
%added; comment-12
Now apply Theorem \ref{power}. 
\end{proof}

It might be  interesting to consider sufficient conditions for $E[(\log X)^n] = E[\log X]^n$ for every $n \in \mathbb{N}$.  
It is not sufficient that $E[(\log X)^2] = E[\log X]^2$. 
For example, if we consider the expectation with respect to 
$$\mu = \frac{1}{3} \left( \delta_{\{-1\}} + \delta_{\{\exp(\pi/\sqrt{3})\}} + \delta_{\{\exp(-\pi/\sqrt{3})\}}  \right),$$ 
then, we have that $E[(\log X)^2] = E[\log X]^2$.

We can also give a characterization for the mixture Cauchy model. 
If the probability density function is given by 
\[ \frac{1-t}{\pi} \frac{\sigma_1}{(x-\mu_1)^2 + \sigma_1^2} + \frac{t}{\pi} \frac{\sigma_2}{(x-\mu_2)^2 + \sigma_2^2}, \]
for some $0 < t < 1$ and $(\mu_1, \sigma_1) \ne (\mu_2, \sigma_2)$, 
then, we call the model the mixture Cauchy model $C\left(t; \mu_1 + \sigma_1 i;  \mu_2 + \sigma_2 i\right)$. 
See Lehmann \cite[pp.480-481]{Lehmann1999} for mixture models of location-scale families. 

We can show the following in the same manner as in the proof of Theorem \ref{power}. 
\begin{Cor}
Let $X$ be a real-valued random variable such that $E\left[|X|^{\delta}\right] < \infty$  for some $\delta > 0$. 
If $E[X^a] = (1-t) \gamma_1^a + t \gamma_2^a, \  a \in D$ for  a subset $D \subset (0, \delta)$ having a limit point in $(0, \delta)$ and some $t \in (0,1)$ and $\gamma_1, \gamma_2 \in \mathbb{H}$, 
then, $X$ follows the mixture Cauchy model $C(t; \gamma_1; \gamma_2)$. 
\end{Cor}

\noindent{\it Acknowledgements} \ The author appreciates the referee for careful reading of the manuscript and giving helpful comments. 
The author was supported by JSPS KAKENHI  19K14549.

\bibliographystyle{amsplain}
\bibliography{GFT-Cauchy}

\end{document}